\documentclass[11pt]{article}

\usepackage[T1]{fontenc}
\usepackage[utf8]{inputenc}
\usepackage[margin=2.54cm]{geometry}
\usepackage{microtype}
\usepackage{amsmath,amssymb,amsthm}
\numberwithin{equation}{section}
\allowdisplaybreaks
\usepackage{enumitem}
\usepackage{xcolor}
\usepackage{aliascnt}
\usepackage{sectsty}
\sectionfont{\centering}
\usepackage[backref=section]{hyperref}
\definecolor{mycitegreen}{HTML}{00A99D}
\hypersetup{
  colorlinks=true,
  linkcolor={blue!85!black},
  citecolor={red!85!black},
  urlcolor=mycitegreen  
}
\usepackage[noabbrev]{cleveref}

\theoremstyle{plain}
\newtheorem{theorem}{Theorem}[section]
\newaliascnt{lemma}{theorem}
\newtheorem{lemma}[lemma]{Lemma}
\aliascntresetthe{lemma}
\newaliascnt{corollary}{theorem}
\newtheorem{corollary}[corollary]{Corollary}
\aliascntresetthe{corollary}
\newaliascnt{conjecture}{theorem}
\newtheorem{conjecture}[conjecture]{Conjecture}
\aliascntresetthe{conjecture}
\newaliascnt{claim}{theorem}
\newtheorem{claim}[claim]{Claim}
\aliascntresetthe{claim}
\theoremstyle{definition}
\newaliascnt{question}{theorem}
\newtheorem{question}[question]{Question}
\aliascntresetthe{question}
\crefname{theorem}{Theorem}{Theorems}
\crefname{lemma}{Lemma}{Lemmas}
\crefname{corollary}{Corollary}{Corollaries}
\crefname{conjecture}{Conjecture}{Conjectures}
\crefname{question}{Question}{Questions}
\crefname{claim}{Claim}{Claims}
\crefname{section}{Section}{Sections}
\Crefname{section}{Section}{Sections}

\title{Nearly balanced spanning subdivisions in dense digraphs}
\author{
Zhilan Wang\thanks{School of Mathematics, Shandong University, Jinan 250100, China. Supported by the National Natural Science Foundation of China (No.~12571373).}
\and
Shuo Wei\footnotemark[1] \thanks{Corresponding author. Email: \href{mailto:shuowei@mail.sdu.edu.cn}{\texttt{shuowei@mail.sdu.edu.cn}}}
}
\date{}

\begin{document}
\maketitle
\vspace{-1.2em}

\begin{abstract}
\noindent
Pavez-Sign\'e [\emph{Combin. Probab. Comput.} 33 (2024), 121--128]
conjectured a Dirac-type condition for spanning $H$-subdivisions and
asked whether the subdivision paths can additionally be required to
have similar lengths. Lee [\emph{European J. Combin.} 124 (2025),
104059] resolved the existence conjecture in the stronger setting of
digraphs. We answer the length-control question in this stronger directed setting:
for every $\varepsilon>0$, there exists a constant $C_0>0$ such that,
for every digraph $H$ with $h$ arcs and no isolated vertices, every $n$-vertex 
digraph $D$ with $n\ge C_0h$ and $\delta^0(D)\ge(1/2+\varepsilon)n$ contains
a spanning $H$-subdivision whose subdivision paths have lengths differing by at most one.
\end{abstract}



\section{Introduction}\label{sec:intro}

Dirac's theorem~\cite{dirac1952} states that every $n$-vertex graph with
minimum degree at least $n/2$ contains a Hamilton cycle. Its classical
digraph analogue is due to Ghouila-Houri~\cite{ghouila-houri}: every
$n$-vertex digraph with minimum semi-degree at least $n/2$ contains a
directed Hamilton cycle. These results initiated a broad theory of
Dirac-type conditions for spanning structures.

Given a digraph $H$, an \emph{$H$-subdivision} is obtained by replacing
each arc $uv\in A(H)$ with a directed $u$--$v$ path so that the
subdivision paths are internally vertex-disjoint and their internal
vertices avoid the vertices corresponding to $V(H)$. The latter
vertices are called \emph{branch vertices}. An $H$-subdivision in a digraph $D$ is \emph{spanning} if it contains every vertex of $D$. It is \emph{balanced} if all subdivision paths have the same length, and \emph{nearly balanced} if the lengths of
any two subdivision paths differ by at most one. In the undirected setting, length control in 
subdivisions has been studied extensively. Thomassen~\cite{thomassen1984} conjectured that sufficiently large
average degree forces a balanced subdivision of every fixed clique.
This conjecture was resolved by Liu and Montgomery~\cite{liu-montgomery}.
Subsequently, Luan, Tang, Wang and Yang~\cite{luan2023} and
Gil Fern\'andez, Hyde, Liu, Pikhurko and Wu~\cite{gilfernandez2023}
independently determined the correct quadratic order of the average
degree required to force a balanced clique subdivision.
More recently, Kim, Liu, Tang, Wang, Yang and
Yang~\cite{kim-length-2026} obtained a linear-in-$e(H)$ average-degree
bound for balanced subdivisions of arbitrary graphs $H$. These results,
however, do not concern spanning subdivisions.

 Pavez-Sign\'e~\cite{Pavez} initiated the study of spanning subdivisions under Dirac-type minimum degree conditions. He first proved a spanning subdivision theorem for $d$-regular graphs with $\log n\le d<n$, and then conjectured that the regularity assumption could be removed.
\begin{conjecture}[Pavez-Sign\'e, Conjecture~3.1 in \cite{Pavez}]
\label{conj}
For every $\varepsilon>0$, there exists a constant $C_0>0$ such that,
for every $C\ge C_0$ and every positive integer $h$, the following
holds. If $G$ is a graph on $n=Ch$ vertices with
$\delta(G)\ge(1+\varepsilon)n/2$, then $G$ contains a spanning
$H$-subdivision for every graph $H$ with $h$ edges and no isolated
vertices.
\end{conjecture}

The proof of Pavez-Sign\'e's result also yields a balanced subdivision
covering all but a linear proportion of the host vertices. This led him
to ask whether a similar length control can be achieved for a spanning
subdivision.

\begin{question}[Pavez-Sign\'e, Question~3.2 in \cite{Pavez}]
\label{ques:pavez}
In the setting of \cref{conj}, can all subdivision paths be
required to have similar lengths?
\end{question}

The existence problem in \cref{conj} was subsequently resolved by
Lee~\cite{Lee} in the stronger setting of digraphs. More precisely, for
every $\varepsilon>0$ there is a constant $C>0$ such that every
$n$-vertex digraph $D$ with $n\ge Ch$ and
$\delta^0(D)\ge(1/2+\varepsilon)n$ contains a spanning $H$-subdivision
for every $h$-arc digraph $H$ with no isolated vertices. More recently,
Wang, Cheng and Yan~\cite{wcy} determined the exact
minimum semi-degree threshold in this setting, proving that
$\delta^0(D)\ge(n+h)/2-1$ is sufficient when $n$ is sufficiently large
compared with $h$.  Neither result, however, controls the lengths of
the individual subdivision paths.

In a different direction, Pavez-Sign\'e, Lee and Petrov~\cite{plp} recently obtained spanning nearly balanced clique subdivisions in pseudorandom graphs. These results do not provide length control for spanning subdivisions of arbitrary digraphs.

We answer \cref{ques:pavez} in the stronger setting of digraphs, and
in fact obtain the strongest possible general form of balance.

\begin{theorem}\label{thm:main}
For every $\varepsilon>0$, there exists a constant $C_0>0$ such that the
following holds. Let $H$ be a digraph with $h\ge1$ arcs and no isolated
vertices, and let $D$ be an $n$-vertex digraph with $n\ge C_0h$ and
$\delta^0(D)\ge(1/2+\varepsilon)n$. Then $D$ contains a spanning
$H$-subdivision in which the lengths of any two subdivision paths differ
by at most one.
\end{theorem}

As an immediate consequence of \cref{thm:main}, we obtain the corresponding
result for graphs.

\begin{corollary}\label{cor}
For every $\varepsilon>0$, there exists a constant $C_0>0$ such that the
following holds. Let $H$ be a graph with $h\ge1$ edges and no isolated
vertices, and let $G$ be an $n$-vertex graph with $n\ge C_0h$ and
$\delta(G)\ge(1+\varepsilon)n/2$. Then $G$ contains a spanning
$H$-subdivision in which the lengths of any two subdivision paths differ
by at most one.
\end{corollary}

%\begin{proof}
%Replace every edge of $G$ by the two oppositely directed arcs and orient
%the edges of $H$ arbitrarily. Apply \cref{thm:main} with
%$\varepsilon/2$ in place of $\varepsilon$, and then forget the
%orientations.
%\end{proof}

In particular, \cref{cor} gives an affirmative answer to
\cref{ques:pavez}. Moreover, the balance in \cref{thm:main} and
\cref{cor} is best possible in general. If $t=|V(H)|$ and the
subdivision paths have lengths $\ell_1,\ldots,\ell_h$, then every
spanning $H$-subdivision satisfies
$\sum_{i=1}^h\ell_i=n-t+h$, which need not be divisible by $h$.
Thus the subdivision paths cannot in general all have the same length.

The coefficient $1/2$ in the minimum semi-degree condition of
\cref{thm:main} is also asymptotically best possible. To see this, let
$H$ be the directed $2$-cycle and let $D$ be the disjoint union of two
complete symmetric digraphs of orders $\lfloor n/2\rfloor$ and
$\lceil n/2\rceil$. Then
$\delta^0(D)=\lfloor n/2\rfloor-1$, while $D$ contains no spanning
$H$-subdivision, since such a subdivision would form a directed
Hamilton cycle.

The proof of \cref{thm:main} revisits the probabilistic partitioning
philosophy used by Pavez-Sign\'e.  We reserve one vertex-disjoint arc
for each arc of $H$, split each subdivision path
into two nearly equal parts, and partition the remaining vertices
accordingly while preserving minimum semi-degree above one half in each
part. Hamilton paths with prescribed endpoints then complete the required
spanning $H$-subdivision.

\medskip
\noindent\textbf{Organization of the paper.}
In \cref{sec2}, we introduce the notation and auxiliary tools. In
\cref{sec3}, we prove the equitable partition lemma, \cref{thm:main},
and \cref{cor}. We conclude with a brief discussion in \cref{sec4}. 

\section{Preliminaries}\label{sec2}
\subsection{Notation}\label{subsec21}

For notation not defined in this paper, we refer the reader to \cite{Joergen}. 
Throughout the paper, all digraphs are finite and loopless. For a digraph $D$, we
write $V(D)$ and $A(D)$ for its vertex set and arc set, respectively.
For $v\in V(D)$, let $N_D^+(v)$ and $N_D^-(v)$ denote the out- and
in-neighborhoods of $v$, and let $d_D^+(v)$ and $d_D^-(v)$ denote the
corresponding degrees. For $X\subseteq V(D)$, write
$N_D^\pm(v,X)=N_D^\pm(v)\cap X$ and
$d_D^\pm(v,X)=|N_D^\pm(v,X)|$. We omit the subscript $D$ when the
ambient digraph is clear.

We write
$\delta^+(D)=\min_{v\in V(D)}d_D^+(v)$ and
$\delta^-(D)=\min_{v\in V(D)}d_D^-(v)$, and define the minimum
semi-degree of $D$ by
$\delta^0(D)=\min\{\delta^+(D),\delta^-(D)\}$.
For $X\subseteq V(D)$, let $D[X]$ denote the subdigraph of $D$ induced
by $X$. When $\sigma\in\{+,-\}$, the notation $N_D^\sigma$ and
$d_D^\sigma$ has the obvious meaning. The order of a directed path is
its number of vertices, and its length is its number of arcs.

\subsection{Auxiliary tools}\label{subsec22}

We shall use the standard concentration bound for the hypergeometric
distribution.
\begin{lemma}[{\cite[Theorem~2.10]{jlr}}]\label{prob}
Let $X$ be a hypergeometric random variable with sample size $m$ and
mean $\mu$. Then, for every $a>0$,
$$
\mathbb P(|X-\mu|\ge a)\le 2\exp\left(-\frac{2a^2}{m}\right).
$$
\end{lemma}

We also use the following theorem of Ghouila-Houri~\cite{ghouila-houri}.

\begin{theorem}[Ghouila-Houri]\label{directedDirac}
Every $m$-vertex digraph $D$ with $\delta^0(D)\ge m/2$ contains a
directed Hamilton cycle.
\end{theorem}

The next consequence gives the Hamilton paths required later.
%We include the proof to keep the dependence on the degree surplus explicit.

\begin{lemma}\label{directedHP}
Let $G$ be an $m$-vertex digraph with $m\ge4$ and
$\delta^0(G)\ge(m+3)/2$. Then, for every two distinct vertices
$x,y\in V(G)$, the digraph $G$ contains a directed Hamilton path from
$x$ to $y$.
\end{lemma}

\begin{proof}
Fix distinct $x,y\in V(G)$. Form a digraph $G'$ from
$G-\{x,y\}$ by adding a new vertex $z$. For every 
$w\in V(G)\setminus\{x,y\}$, add the arc $zw$ whenever $xw\in A(G)$
and the arc $wz$ whenever $wy\in A(G)$. Thus $G'$ has $m-1$ vertices.
Every vertex different from $z$ loses at most two in-neighbors and at
most two out-neighbors when passing from $G$ to $G'$, while
$d_{G'}^+(z)\ge d_G^+(x)-1$ and $d_{G'}^-(z)\ge d_G^-(y)-1$.
Consequently, $\delta^0(G')\ge(m-1)/2$. By
\cref{directedDirac}, $G'$ contains a directed Hamilton cycle.
Write the segment of this cycle through $z$ as $wzu$. Deleting $z$
leaves a directed Hamilton path from $u$ to $w$ in
$G-\{x,y\}$. By the definition of $G'$, the arcs $xu$ and $wy$ lie in
$G$, so adjoining them gives a directed Hamilton path from $x$ to $y$.
\end{proof}

\section{Proofs of Theorem \ref{thm:main} and Corollary \ref{cor}}\label{sec3}

We begin with a partition lemma that will be used in the proof of
\cref{thm:main}. It is a directed version of the recursive partitioning
argument in~\cite[Lemma~2.4]{Pavez}. For each part, we also prescribe
two vertices outside the partition whose in-degrees  and out-degrees into that
part are required to remain large. The prescribed vertices for
different parts need not be distinct.

\begin{lemma}\label{partitionlemma}
For every $\rho>0$, there exists an integer $L=L(\rho)$ such that the
following holds. Let $D$ be a digraph, let $R\subseteq V(D)$, and let
$I$ be a finite index set. For each $i\in I$, let
$p_i,q_i\in V(D)\setminus R$, where the vertices prescribed for
different indices need not be distinct. Suppose that, for every
$z\in R\cup\{p_i,q_i:i\in I\}$ and every $\sigma\in\{+,-\}$, we have
$d_D^\sigma(z,R)\ge(1/2+\rho)|R|$. Let $(r_i)_{i\in I}$ be positive
integers with $\sum_{i\in I}r_i=|R|$, $r_i\ge L$ for every $i$, and
$|r_i-r_j|\le1$ for all $i,j\in I$. Then $R$ has a partition
$R=\bigcup_{i\in I}R_i$ with $|R_i|=r_i$ such that, for every
$i\in I$, every $z\in R_i\cup\{p_i,q_i\}$ and every
$\sigma\in\{+,-\}$, $d_D^\sigma(z,R_i)\ge(1/2+\rho/2)r_i$.
\end{lemma}

\begin{proof}
Set $\theta=1/2+\rho$. We choose $L$ sufficiently large in terms of
$\rho$. For a non-empty set $P\subseteq I$ and a set $S\subseteq R$
with $|S|=\sum_{i\in P}r_i$, call $(S,P)$ \emph{good} if
\begin{equation}\label{eq:partition-invariant}
 d_D^\sigma(z,S)\ge\bigl(\theta-2|S|^{-1/4}\bigr)|S|
\end{equation}
for every $z\in S\cup\{p_i,q_i:i\in P\}$ and every
$\sigma\in\{+,-\}$. The pair $(R,I)$ is good by assumption.

We first show that every good pair involving at least two indices can
be split into two smaller good pairs.

\begin{claim}\label{clm}
If $(S,P)$ is good and $|P|\ge2$, then there are non-empty sets
$P_1,P_2$ partitioning $P$ and sets $S_1,S_2$ partitioning $S$ such
that $|S_j|=\sum_{i\in P_j}r_i$ and $(S_j,P_j)$ is good for
$j\in\{1,2\}$.
\end{claim}

\begin{proof}\renewcommand*{\qedsymbol}{$\blacksquare$}
Write $s=|S|$. Partition $P$ into non-empty sets $P_1,P_2$ whose
sizes differ by at most one, and set
$s_j=\sum_{i\in P_j}r_i$. Since the $r_i$ differ pairwise by at most
one, by taking $L\ge4$ we may assume that
$s/4\le s_j\le3s/4$ for $j\in\{1,2\}$. Choose $S_1$ uniformly among
the $s_1$-subsets of $S$ and set $S_2=S\setminus S_1$.
orders
Fix $j\in\{1,2\}$, $\sigma\in\{+,-\}$ and
$z\in\{p_i,q_i:i\in P_j\}$. Then $d_D^\sigma(z,S_j)$ is
hypergeometric with mean at least
$(\theta-2s^{-1/4})s_j$ by \eqref{eq:partition-invariant}. The
difference between this lower bound and the target
$(\theta-2s_j^{-1/4})s_j$ is
$2(s_j^{-1/4}-s^{-1/4})s_j$. Since $s_j\le3s/4$, this difference is at
least $c s_j^{3/4}$ for an absolute constant $c>0$.

Now let $z\in S$. Conditional on $z\in S_j$, the set
$S_j\setminus\{z\}$ is a uniformly chosen $(s_j-1)$-subset of
$S\setminus\{z\}$. Hence $d_D^\sigma(z,S_j)$ is hypergeometric with
conditional mean $\frac{s_j-1}{s-1}d_D^\sigma(z,S)$. This differs from
$\frac{s_j}{s}d_D^\sigma(z,S)$ by at most one. Thus, after increasing
$L$ if necessary, the conditional mean exceeds
$(\theta-2s_j^{-1/4})s_j$ by at least $(c/2)s_j^{3/4}$.

By \cref{prob}, there exists a constant $c_1>0$ such that, for
each prescribed vertex, the required degree bound fails with
probability at most $2e^{-c_1s_j^{1/2}}$. For $z\in S$, the same
bound holds for the event that $z\in S_j$ and the required degree
bound fails.

There are at most $4s+4|P|\le8s$ such events altogether. Since
$s_j\ge s/4$, the union bound shows that the probability that either
$(S_1,P_1)$ or $(S_2,P_2)$ is not good is at most
$16s e^{-c_2s^{1/2}}$
for some constant $c_2>0$. Whenever a split is required, $s\ge2L$.
Thus, by choosing $L$ sufficiently large, this probability is less
than one, and hence a good split exists.
\end{proof}

Starting from $(R,I)$, repeatedly apply \cref{clm} until
every index set is a singleton. This gives a partition
$R=\bigcup_{i\in I}R_i$ with $|R_i|=r_i$. By
\eqref{eq:partition-invariant}, every
$z\in R_i\cup\{p_i,q_i\}$ and $\sigma\in\{+,-\}$ satisfies
$d_D^\sigma(z,R_i)\ge(\theta-2r_i^{-1/4})r_i$. Finally, increase $L$
so that $2L^{-1/4}\le\rho/2$. Since $\theta=1/2+\rho$, the required
bound follows.
\end{proof}

We are now in a position to prove \cref{thm:main}.

\begin{proof}[Proof of \cref{thm:main}]
We may assume that $0<\varepsilon<1/2$ and set
$\rho=\varepsilon/2$. Let $L=L(\rho)$ be given by
\cref{partitionlemma}, enlarged if necessary so that $\rho L\ge5$,
and choose $C_0\ge\max\{4,8/\varepsilon,2L+6\}$.

Let $H$ and $D$ satisfy the assumptions of the theorem, and write
$t=|V(H)|$. Since $H$ has no isolated vertices, $t\le2h$. Write
$A(H)=\{e_1,\ldots,e_h\}$. A spanning $H$-subdivision whose subdivision
paths have lengths $\ell_1,\ldots,\ell_h$ must satisfy
$\sum_{i=1}^h\ell_i=n-t+h$. Choose positive integers
$\ell_1,\ldots,\ell_h$ with this sum and with
$|\ell_i-\ell_j|\le1$ for all $i,j$. Since
$(n-t+h)/h\ge n/h-1\ge C_0-1$, every $\ell_i$ is at least $C_0-1$.

As $\delta^0(D)\ge n/2$, \cref{directedDirac} gives a directed
Hamilton cycle in $D$. Since $n\ge C_0h$ and $C_0\ge4$, this cycle
contains $h$ pairwise vertex-disjoint arcs. For each
$e=uv\in A(H)$, denote the corresponding arc by $x_ey_e$, where
$x_e$ is assigned to the tail $u$ and $y_e$ to the head $v$. Let
$M=\{x_ey_e:e\in A(H)\}$. Since $t\le2h$, we may choose distinct vertices
$b_v\in V(D)$, $v\in V(H)$, outside $V(M)$ to serve as the branch
vertices corresponding to the vertices of $H$.

For every arc $e=uv$, choose positive integers $a_e^-$ and $a_e^+$
with $a_e^-+a_e^++1=\ell_e$ and $|a_e^--a_e^+|\le1$. Set
$r_e^-=a_e^--1$ and $r_e^+=a_e^+-1$. Because the $\ell_e$ take two
consecutive values and $\ell_e-1$ is split as evenly as possible, the
$2h$ integers $r_e^-,r_e^+$ differ pairwise by at most one. The choice
$C_0\ge2L+6$ also gives $r_e^-,r_e^+\ge L$ for every $e\in A(H)$.

Let $B=\{b_v:v\in V(H)\}$ and set $R=V(D)\setminus(B\cup V(M))$. Then $|R|=n-t-2h$, while
$$\sum_{e\in A(H)}(r_e^-+r_e^+) =\sum_{e\in A(H)}(a_e^-+a_e^+-2)=\sum_{e\in A(H)}(\ell_e-3) =n-t-2h =|R|.$$
Thus the numbers $r_e^-,r_e^+$ prescribe a partition of all vertices of
$R$.

For every $z\in V(D)$ and every $\sigma\in\{+,-\}$, using
$|B|+|V(M)|=t+2h\le4h$, we have
\begin{align*}
d_D^\sigma(z,R)
 &\ge \delta^0(D)-|B|-|V(M)|\\
 &\ge \left(\frac12+\varepsilon\right)n-4h
 \ge \left(\frac12+\frac\varepsilon2\right)n
 \ge \left(\frac12+\rho\right)|R|.
\end{align*}
Here the third inequality follows from $n\ge C_0h$ and
$C_0\ge8/\varepsilon$.

Apply \cref{partitionlemma} with index set
$I=\{(e,-),(e,+):e\in A(H)\}$. For each arc $e=uv$, assign
$b_u,x_e$ to $(e,-)$ with prescribed size $r_e^-$, and $y_e,b_v$ to
$(e,+)$ with prescribed size $r_e^+$. We obtain a partition
$R=\bigcup_{e\in A(H)}(R_e^-\cup R_e^+)$, where
$|R_e^-|=r_e^-$ and $|R_e^+|=r_e^+$, such that, for every
$\sigma\in\{+,-\}$,
$d_D^\sigma(z,R_e^-)\ge(1/2+\rho/2)r_e^-$ for every
$z\in R_e^-\cup\{b_u,x_e\}$, and
$d_D^\sigma(z,R_e^+)\ge(1/2+\rho/2)r_e^+$ for every
$z\in R_e^+\cup\{y_e,b_v\}$.

Fix an arc $e=uv\in A(H)$ and set
$W_e^-=R_e^-\cup\{b_u,x_e\}$. Writing $r=r_e^-$, we have
$|W_e^-|=r+2$ and
$$
\delta^0(D[W_e^-])
\ge\left(\frac12+\frac{\rho}{2}\right)r
\ge\frac{r+5}{2}
=\frac{|W_e^-|+3}{2},
$$
where the second inequality follows from $r\ge L$ and $\rho L\ge5$.
By \cref{directedHP}, $D[W_e^-]$ contains a directed
Hamilton path $P_e^-$ from $b_u$ to $x_e$. Its length is
$r_e^-+1=a_e^-$. Similarly,
$D[R_e^+\cup\{y_e,b_v\}]$ contains a directed Hamilton path $P_e^+$
from $y_e$ to $b_v$ of length $r_e^++1=a_e^+$.

Concatenating $P_e^-$, the arc $x_ey_e$, and $P_e^+$ gives a directed
$b_u$--$b_v$ path of length $\ell_e$. Doing this for every
$e\in A(H)$ gives an $H$-subdivision. Indeed, the sets
$R_e^\pm$, $e\in A(H)$, are pairwise disjoint, the arcs of $M$ are
pairwise vertex-disjoint, and $R\cup V(M)$ is disjoint from $B$.
Hence the resulting paths are internally vertex-disjoint and intersect
only at common branch vertices. Since
$V(D)=B\mathbin{\dot\cup}V(M)\mathbin{\dot\cup}R$, the subdivision is
spanning. Finally, the choice of the $\ell_e$ implies that its
subdivision path lengths differ by at most one.
\end{proof}

We finish this section with the proof of \cref{cor}.

\begin{proof}[Proof of \cref{cor}]
Orient the edges of $H$ arbitrarily and replace every edge of $G$ by
the two oppositely directed arcs. The resulting digraph has minimum
semi-degree $\delta(G)$. The result now follows from \cref{thm:main},
applied with $\varepsilon/2$ in place of $\varepsilon$.
\end{proof}

\section{Concluding remarks}\label{sec4}

Our result gives a nearly balanced strengthening of the spanning
subdivision theorem in the Dirac-type setting and, in particular,
answers Question~3.2 of Pavez-Sign\'e \cite{Pavez} in the stronger setting of
digraphs. It is natural to ask whether the $\varepsilon n$ surplus 
in the minimum semi-degree condition can be replaced by the optimal 
additive term. In view of the exact spanning subdivision theorem
of Wang, Cheng and Yan~\cite{wcy},the following question seems 
particularly natural.

\begin{question}\label{ques}
Does there exist a constant $C_0>0$ such that, for every digraph $H$
with $h$ arcs and no isolated vertices and every $n$-vertex digraph
$D$ with $n\ge C_0h$ and $\delta^0(D)\ge(n+h)/2-1$, the digraph $D$ contains a spanning
$H$-subdivision whose subdivision paths have lengths differing by at most one?
\end{question}

\noindent\textbf{\large Acknowledgment.}
The authors used ChatGPT 5.6 to assist in the development of the
probabilistic partition argument in \cref{partitionlemma}, as well as
for some language polishing. All mathematical arguments were
independently verified by the authors, who take full responsibility
for the content of this work.

\end{document}